\documentclass[preprint,12pt]{elsarticle}
\usepackage{geometry} 
\usepackage{amssymb}
\usepackage{amsthm}
\usepackage{amsmath}
\usepackage{graphicx}
\usepackage{comment}
\usepackage{mathrsfs}
\usepackage{srcltx}

\newtheorem{theorem}{Theorem}

\newtheorem{definition}{Definition}
\newtheorem{cons}{Corollary}

\journal{Journal of Mathematical Analysis and Applications}

\begin{document}

	\begin{frontmatter}
		
		
		
		\title{Area functional and majorant series estimates in the class of bounded functions in the disk}
		
		
		\author{R.S. Khasyanov\fnref{label2}}
		\ead{st070255@student.spbu.ru}
		\affiliation[label2]{organization={Saint Petersburg State University},
			addressline={Universitetskii prosp., 28D}, 
			city={Saint Petersburg},
			postcode={198504}, 
			country={Russian Federation}}
		
		\begin{abstract}
		In this article, the new inequalities for  the weighted sums of coefficients in the class of bounded functions in the disk are obtained. We develop the methods of I.R.~Kayumov and S.~Ponnusamy, using E.~Reich's theorem on the majorization of subordinate functions. The sharp estimates for the area of the image of the disk of radius  $r$ under the action of the function which is expanded into a lacunary series of standard form are obtained. Under significantly lower than in \cite{Khas} restrictions on the initial coefficient, the estimates for the Bohr--Bombieri function of the Hadamard convolution operator are proved. Using the example of the differentiation operator, it is shown that in some cases the new method for calculating the lower bound for the Bohr radius of the Hadamard operator with a fixed initial coefficient is more effective than the known one.

	\end{abstract}
		
		
		
	\end{frontmatter}
	
	
	\section{Introduction} \label{intro}
	Let $\mathbb{D}=\{|z|<1\},$ $f(z)=\sum_{n\ge 0}a_nz^n, \: z\in \mathbb{D},$ $\|f\|_{\infty}=\sup_{z\in \mathbb{D}}|f(z)|$ and $c_n\ge 0,\: n\ge 0.$ In this article we study the  inequalities for the following functionals in the class of bounded analytic functions in the disk:
	
	$$\sum_{n\ge m}c_n|a_n|^2r^{2n} \:\: \text{and} \:\: \sum_{n\ge m}c_n|a_n|r^n, \quad 0\le r<1, \: m\in \mathbb{N}_0.$$
In particular, example  of these sums is the area functional (area of $f(r\mathbb{D})$ taking account of multiplicity)
$$S_rf=\int_{r\mathbb{D}}|f^{\prime}(z)|^2dxdy=\pi\sum_{n\ge 1}n|a_n|^2r^{2n}.$$
The estimates of this functional will be given in the Section \ref{Sec3}. Another examples of these functionals is the $L^2$--norm of the analytic function on the circle $$\|f(rz)\|_2^2=\dfrac{1}{2\pi}\int_{0}^{2\pi }|f(re^{i\theta})|^2d\theta=\sum_{n\ge 0}|a_n|^2r^{2n}$$
and majorant series
$$M_rf=\sum_{n\ge 0}|a_n|r^n.$$

Majorant series is related to the classical Bohr inequality. In 1914, H. Bohr, studying the problems of absolute convergence of Dirichlet series, noticed the following interesting fact, which is now called the Bohr phenomenon:

\newtheorem*{A}{Theorem A}
\begin{A} \label{Bohr}
		Let $f(z)=\sum_{n\ge 0}a_nz^n, \:z\in \mathbb{D}$ and $\|f\|_{\infty}\le1$. Then
		$$\sum_{n\ge 0}^{} |a_n| r^n\le1, \quad
		0\le r\le 1/3.$$
		The number $1/3$ is the best possible.
	\end{A} 
Bohr proved this theorem only for $r\le 1/6.$ The constant $1/3$ was obtained in the same year independently by M. Riesz, I. Schur and F. Wiener (various proofs are collected in the appendix of the work \cite{Dixon05}). This theorem is equivalent to the inequality known as the Bohr inequality:

	$$
	\sum_{n\ge0}|a_n|r^n\le \|f\|_\infty, \quad  0\le r\le 1/3.
	$$   
	
The active study of various generalizations and modifications of the Bohr inequality began in the middle of 1990s, since P. Dixon, using the Bohr inequality, solved the long-standing problem on the characterization of Banach algebras \cite{Dixon05}. Part of  subsequent research in this area
is directed towards extending the Bohr phenomenon for functions of one variable (\cite{AKP}-\cite{Boas}, \cite{Bombb}, \cite{Bomb}, \cite{IKP}---\cite{KayumpImpr}, \cite{PVW}). The other part is directed towards generalizations of the Bohr phenomenon in multidimensional spaces and for more abstract cases (see, for example, \cite{Aizen}, \cite{BPS}, \cite{Boas}, \cite{BoasKhav}, \cite{Ann.Math.}, \cite{PaSi}, \cite{PPS}). In 2018, B. Bhowmik and N. Das applied Bohr's inequality to the
question of comparing majorant series of subordinate functions \cite{BhowDas}. In the works \cite{KayumpImpr}, \cite{PVW} estimates of the sum of majorant series and the area functional of the analytic function were obtained. For studies related to Bohr's inequality, see \cite{Aizen}-\cite{Ann.Math.}, \cite{IKP}-\cite{PPS}, \cite{Ricci}.
	 In Section 4, we will give the new estimates of the weighted majorant series in the class of bounded functions.
	
\vspace{7mm}
	\section{Majorization of subordinate functions}

To estimate the functionals from the Section \ref{intro} we will use theorems on the majorization of subordinate functions. Let us first recall the definition of subordinate functions:

	\begin{definition}
		
	We say that the analytic in  $\mathbb{D}$ function $f(z)$ is subordinated to the analytic in  $\mathbb{D}$ function $g(z)$ and write   $f\prec g$ if there exists an analytic in  $\mathbb{D}$ function  $\omega(z)$  such that $|\omega(z)|\le 1, z\in \mathbb{D},$ $\omega(0)=0 $ and
	$$f(z)=g(\omega(z)), \quad z\in \mathbb{D}.$$

	\end{definition}
	Note that if $g$ is univalent function then
	$$f \prec g \quad \Longleftrightarrow \quad f(0)=g(0) \:\: \text{and} \:\: f(\mathbb{D})\subset g(\mathbb{D}).$$
	
	In 1925, J. E.~Littlewood proved \cite{Little} the theorem on the majorization of norms in Lebesgue spaces of subordinate functions:
		
\newtheorem*{B}{Theorem B}
\begin{B}
			Let $f \prec g$ in $\mathbb{D}.$ Then for $p\ge 0,$ 
		$$\int_{0}^{2\pi}|f(re^{i\theta})|^pd\theta\le \int_{0}^{2\pi}|g(re^{i\theta})|^pd\theta, \quad  0\le r \le 1.$$
	\end{B}

For $p=2$, it follows from this theorem and from Parseval’s equality that
if $f(z)=\sum_{n\ge 0}a_nz^n, \: \: g(z)=\sum_{n\ge 0}b_nz^n$ and $f \prec g$ in $\mathbb{D},$ then
		$$\sum_{n\ge 0}|a_n|^2r^{2n}\le \sum_{n\ge 0}|b_n|^2r^{2n}, \quad 0\le r \le 1.$$

In 1951, G. Goluzin obtained a weighted analogue of  this inequality:
		
		\vspace{2mm}
	
	\newtheorem*{C}{Theorem C}
	\begin{C}[\cite{Gol}]\label{thGoluzin}
		Let $f(z)=\sum_{n\ge 0}a_nz^n, \: g(z)=\sum_{n\ge 0}b_nz^n$  and $f\prec g$ in $\mathbb{D}.$ Then for any non-increasing sequence $\lambda_n\ge 0,$
		\begin{equation*} \label{2}
			\sum_{n\ge 1}\lambda_n|a_n|^2\le  \sum_{n\ge 1}\lambda_n|b_n|^2.
		\end{equation*}
	\end{C}

\vspace{2mm}
From this theorem Goluzin obtained inequalities for the areas of the images of analytic functions, as well as the inequality for the norms of derivatives in the space $L^2$.
The following statement follows from Theorem C:

\vspace{2mm}
	\newtheorem*{0}{Proposition 1} \label{0}
	\begin{0}
		Let $c_n>0, \: n\ge 1,\:$ $f(z)=\sum_{n\ge 0}a_nz^n, \: g(z)=\sum_{n\ge 0}b_nz^n$  and $f\prec g$ in $\mathbb{D}.$ Then
		$$\sum_{n\ge 1}c_n|a_n|^2r^{2n}\le \sum_{n\ge 1}c_n|a_n|^2r^{2n}, \quad r^2\le \inf_{n\ge 1}\dfrac{c_n}{c_{n+1}}.$$
	\end{0}

\vspace{2mm}

Let us fix the sequence $c_n>0, \: n\ge 1.$ We define the function $G(r)$ as follows: $$G(r):=\sup_{f\prec g} \dfrac{\sum_{n\ge 1}c_n|a_n|^2r^{2n}}{\sum_{n\ge 1}c_n|b_n|^2r^{2n}},$$ where $f(z)=\sum_{n\ge 0}a_nz^n, \: g(z)=\sum_{n\ge 0}b_nz^n.$ 

\vspace{4mm}
The following statement generalizes theorem from [\cite{Reich}, P. 260--261] (we obtain its proof  in a similar way):

\vspace{2mm}
\begin{theorem}\label{G(r)}
	Let $r\ge 0,\:$  $s\ge 2, \: c_n>0 \:(n\ge 1),$ $c_1=1$ and
	\begin{equation}\label{Cond1}
		r^2\le c_nr^{2n} \le c_sr^{2s}, \quad 1\le n<s.
		\end{equation}
	
	Then $$G(r)={c_s}r^{2s-2}, \quad \inf_{n\ge s-1}\dfrac{c_n}{c_{n+1}}\le r^2 \le \inf_{n\ge s}\dfrac{c_n}{c_{n+1}}.$$ The supremum of $G(r)$ is attained at the following functions
	
	\vspace{2mm}
	(i) If $r^2<\inf_{n\ge 1}\dfrac{c_n}{c_{n+1}},$ then $\omega(z)=e^{i\phi} z$;
	
	\vspace{2mm}
	(ii) If $r^2=\inf_{n\ge 1}\dfrac{c_n}{c_{n+1}}$, then
	$$
	\omega(z)=e^{i\phi} z
	$$
	or
	$$
	g(z)=C_0+C_1 z, \quad f(z)=C_0+e^{i\phi}C_1  z^2 ;
	$$
	\hspace{5mm}(iii) If $ \inf_{n\ge s-1}\dfrac{c_n}{c_{n+1}}<r^2<\inf_{n\ge s}\dfrac{c_n}{c_{n+1}}$
	and $s \geq 2$, then
	$$
	g(z)=C_0+C_1 z, \quad f(z)=C_0+e^{i\phi}C_1  z^s;
	$$
	\hspace{4mm} (iv) If $r^2=\inf_{n\ge s}\dfrac{c_n}{c_{n+1}}$
	and $
	s \geq 2
	$, then
	$$
	g(z)=C_0+C_1 z, \quad  f(z)=C_0+e^{i\phi}C_1  z^s
	$$
	or
	$$
	g(z)=C_0+C_1 z, \quad f(z)=C_0+e^{i\phi}C_1 z^{s+1}.
	$$
\end{theorem}

\vspace{3mm}

\textbf{Remark 1.} It is easy to check that if $r^2\ge 1/c_2$ and the sequence $\dfrac{c_{n}}{c_{n+1}}$ is non-decreasing for $n\ge 2,$ then the condition \eqref{Cond1} holds.

\vspace{2mm}
\textbf{Remark 2.} In the work  \cite{Reich} it was shown that if $c_n=n,$ then the extremal functions in Theorem \ref{G(r)} are unique in this problem.

\vspace{2mm}

\textbf{Remark 3.} In the work \cite{AreaQuasi} the last theorem for $c_n=n$ was generalized to the case of quasi-subordinate functions. The asymptotic behavior of the function $G(r), \: r\rightarrow 1,$ for $c_n=n$ was studied in the article \cite{GrowthArea}.

\vspace{3mm}

\begin{proof}[Proof of Theorem \ref{G(r)}]
The sum $\sum_{n\ge 1} c_n\left|a_n\right|^2 r^{2 n}$ can be expressed in the following form:
	\begin{multline*}
		\left[c_s r^{2 s} \sum_{n= 1}^{s-1} |a_n|^2+\sum_{n\ge s} c_n\left|a_n\right|^2 r^{2 n}\right]-\left[c_s r^{2 s} \sum_{n=1}^{s-1}\left|a_n\right|^2-\sum_{n=1}^{s-1} c_n\left|a_n\right|^2 r^{2 n}\right] \\
		\\
		=\sum_{n\ge 1} \lambda_n^{(s)}\left|a_n\right|^2-\sum_{n=1}^{s-1}\left(c_s r^{2 s}-c_n r^{2 n}\right)\left|a_n\right|^2,		
	\end{multline*}
	where
	$$
	\lambda_n^{(s)}= \begin{cases}c_s r^{2 s}, \quad 1 \leq n \leq s-1, \\ c_n r^{2 n}, \quad n \geq s .\end{cases}
	$$
Similarly, 
	$$	\sum_{n\ge 1} c_n\left|b_n\right|^2 r^{2 n} 
	=\sum_{n\ge 1} \lambda_n^{(s)}\left|b_n\right|^2-\sum_{n=1}^{s-1}\left(c_s r^{2 s}-c_n r^{2 n}\right)\left|b_n\right|^2.$$
By the conditions of theorem, $\inf_{n\ge s-1}\dfrac{c_n}{c_{n+1}}\le r^2 \le \inf_{n\ge s}\dfrac{c_n}{c_{ n+1}},$ therefore the sequence $ (\lambda_n^{(s)})_{n\ge 1}$ is non-increasing, so we can use Theorem C:
	$$\sum_{n\ge 1} \lambda_n^{(s)}\left|a_n\right|^2\le \sum_{n\ge 1} \lambda_n^{(s)}\left|b_n\right|^2.$$
Thus,

\begin{flushleft}
$\sum_{n\ge 1} c_n(\left|a_n\right|^2-|b_n|^2) r^{2 n}$
\end{flushleft}
			\begin{equation*}
			\begin{aligned}	\le \sum_{n=1}^{s-1}\left(c_s r^{2 s}-c_n r^{2 n}\right)\left|b_{n}\right|^{2}-\sum_{n=1}^{s-1}\left(c_s r^{2 s}-c_n r^{2 n}\right)\left|a_{n}\right|^{2}
		\le \sum_{n=1}^{s-1}\left(c_s r^{2 s}-c_n r^{2 n}\right)\left|b_{n}\right|^{2}
	\\ \\ =\left(c_s r^{2 s-2}-1\right) \sum_{n=1}^{s-1} c_n\left|b_{n}\right|^{2} r^{2 n}-c_s r^{2 s-2} \sum_{n=1}^{s-1}\left(c_n r^{2 n}-r^{2}\right)\left|b_{n}\right|^{2} 
		\\
		\\
		\le \left(c_s r^{2 s-2}-1\right)	\sum_{n\ge 1} c_n|b_n|^2 r^{2 n}.
	\end{aligned}
	\end{equation*}
The last inequality follows from the condition \eqref{Cond1}. Thus, $\dfrac{\sum_{n \ge 1}c_n|a_n|^2r^{2n}}{\sum_{n \ge 1}c_n|b_n|^2r^{2n}}\le c_sr^{2s-2}$ for all $f$ and $g$ such that $f\prec g.$ Passing to the supremum in this inequality, we obtain  $G(r)\le c_sr^{2s-2}.$
\end{proof}

\vspace{5mm}
\section{Area functional estimates} \label{Sec3}
Using Theorem \ref{G(r)} for $c_n=n, \: n\ge 1,$ E.~Reich obtained the following corollary from it. Let $f(0)=0,$ then
$$S_rf\le \pi sr^{2s} \|f\|_{\infty}^2, \quad \dfrac{s-1}{s}\le r^2 \le \dfrac{s}{ s+1}.$$
This fact becomes obvious if we notice that if $\|f\|_{\infty}\le 1$ and $f(0)=0,$ then $f(z)\prec z$ in the disk $\mathbb {D}.$ We generalize this statement for lacunary series of the form $\sum_{n\ge 0}a_{mn+p}z^{mn+p}, \: m\ge p\ge 0,$ and obtain an interesting corollary for the standard case $m=1, \: p=0.$

\vspace{2mm}
\begin{theorem}\label{Area}
Let $s \in \mathbb{N}_0, \: m\ge p\ge 0, \: f(z)=\sum_{n\ge 0}a_{mn+p}z^{mn+p}$ and $r\ge 0,$  then
$$S_rf\le \pi (ms+p) r^{2(ms+p)} \|f\|_{\infty}^2, \quad \dfrac{m(s-1)+p}{ms+p}\le r^{2m} \le \dfrac{ms+p}{m(s+1)+p}.$$	
The functions $f(z)=Cz^{ms+p}, \: C\in \mathbb{C},$ are extremal.
\end{theorem}

\vspace{2mm}
\begin{cons}
Let $s\in \mathbb{N}$ and $f(z)=\sum_{n\ge 0}a_nz^n,$ then
\begin{equation} \label{Cor1}S_rf\le \pi sr^{2s} \|f\|_{\infty}^2, \quad \dfrac{s-1}{s}\le r^2 \le \dfrac{s}{s+1}.
\end{equation}
All extremal functions are of the form $f(z)=Cz^s, \: C\in \mathbb{C}.$
\end{cons}

The last inequality seems very simple and classical. It would be interesting to obtain its analogues in other classes of analytic functions.

\vspace{2mm}
\begin{cons} \label{cons3}
	Let $f(z)=\sum_{n\ge 0}a_nz^n,$ then
$$
	S_rf\le \pi  \|f\|_{\infty}^2, \quad 0\le  r\le 1/\sqrt[6]{3}.
$$
The number $1/\sqrt[6]{3}$  is the best possible. Extremal functions are of the form $f(z)=Cz^3, \: C\in \mathbb{C}.$
\end{cons}
The inequality from Corollary \ref{cons3} is an analogue of Bohr's inequality (Theorem A) for the area functional.

\vspace{2mm}
\begin{proof}[Proof of Corollary \ref{cons3}]
	Let $g(x)=\dfrac{1}{\sqrt[2x]{x}}$ and $h(x)=\sqrt{\dfrac{x}{x+1}}, \: x\ge 1.$ 
It follows from the inequality \eqref{Cor1} that $$S_rf\le \pi \|f\|^2_{\infty}, \quad r\le \sup_{s\in \mathbb{N}}\{g(s+1): \: g(i)\ge h(i), \: 1\le i\le s-1 \:\: \text{and} \:\: g(s)>h(s)\}.$$  It is easy to check that $g(x)=h(x) \Longleftrightarrow x=y=2.29317...,\:$ $g(x)>h(x), \: 1<x<y,$ and $g(x)<h(x), \: y<x<\infty.$ Thus, the supremum is attained at $s=2.$
\end{proof}

\vspace{3mm}
\begin{proof}[Proof of Theorem \ref{Area}]
Let $\|f\|_{\infty}\le 1.$ It follows from Schwarz lemma that $\big|\sum_{n\ge 0}a_{mn+p}z^{mn}\big|\le 1, \: z\in \mathbb{D}.$ Without loss of generality, assume that $a=a_p>0.$ Let $w=z^m,$ then
	$$\sum_{n\ge 0}a_{mn+p}w^n \prec \dfrac{w+a}{1+{a}w}=a+(a-1/a)\sum_{n\ge 0}(-aw)^n.$$
	Let $\rho=r^m.$ Using Theorem C for $s=0$ and $s=1$ and Theorem \ref{G(r)}  for $s\ge 2$ with weights $c_n=\dfrac{mn+p}{m+p}, \: n\ge 1$ (which satisfy the conditions of Remark 1), we obtain that if
	$$\dfrac{m(s-1)+p}{ms+p}\le \rho^2\le \dfrac{ms+p}{m(s+1)+p}, $$
	then for $s=0$ and $s=1,$
		\begin{equation*}
		\dfrac{S_rf}{\pi}=r^{2p}\cdot \sum_{n\ge 0}(mn+p)|a_{mn+p}|^2\rho^{2n}\le
		r^{2p}\cdot \Big(pa^2+\dfrac{(1-a^2)^2(m+p(1-(a\rho)^2))}{(1-(a\rho)^2)^2}\rho^{2s}\Big),
	\end{equation*}
	and for $s\ge 2,$
	\begin{equation*}
	\dfrac{S_rf}{\pi}=r^{2p}\cdot \sum_{n\ge 0}(mn+p)|a_{mn+p}|^2\rho^{2n}\le
	 r^{2p}\cdot \Big(pa^2+\dfrac{(1-a^2)^2(m+p(1-(a\rho)^2))}{(1-(a\rho)^2)^2}\dfrac{sm+p}{m+p}\rho^{2s}\Big).
	\end{equation*}
Let us denote $x=a^2,$ 
$$I_{0}=\Big[0,   \dfrac{m}{m+p}\Big], \quad I_{s}=\Big[\dfrac{m(s-1)+p}{ms+p}, \dfrac{ms+p}{m(s+1 )+p}\Big],\: s\ge 1.$$ Fix $b\in [0, 1)$ and choose $s\ge 0,$ such that $b\in I_{s}.$ For every fixed $b\in [0,1)$ we find the maximums of the functions

$$
h(x, b)=\begin{aligned}
	\begin{cases}
		px+\dfrac{(1-x)^2(m+p(1-bx))}{(1-bx)^2}b, \quad
		b\in I_{0},
\\  \\
px+\dfrac{(1-x)^2(m+p(1-bx))}{(1-bx)^2}\dfrac{sm+p}{m+p}b^s,  \quad b\in I_{s}, \: s\ge 1.
	\end{cases}
\end{aligned}
$$
for $x\in [0,1].$

\vspace{2mm}
Let $b\in I_{0}.$ Consider the function $g,$ defined as follows:

\begin{flushleft}
$g(x):=(1-bx)^3h_x^{\prime}(x,b)=p(1-bx)^3$
\end{flushleft}
\begin{equation*}
 \begin{aligned}
	+b\big(-(1-x)(1-bx)(2(m+p(1-bx))+pb(1-x))+2b(1-x)^2(m+p(1-bx))\big) \\ \\
  = -6b^2px^2+b(5p+2m+2pb-2mb-pb^2)\cdot x+ p+b(-2m-2p+2mb+pb)
 \end{aligned}
\end{equation*}

Let us calculate the values of the function $g$ at the points $x=1$ and $x=0:$
\begin{multline*}
		g(1)=p(1-b)(b^2+4b+1)\ge 0, \hspace{10mm}
g(0)=(b-1)(b(p+2m)-p)\ge 0 \:\: \Longleftrightarrow  \:\: b \le \dfrac{p}{p+2m}. 
\end{multline*}
	Since the function $g$ is concave, then for $b\le \dfrac{p}{p+2m},$ inequality $h_x^{\prime}(x,b)\ge 0$ holds for all $x\in [0, 1].$ In this case, the function $h$ attains its maximum at the point $x=1.$ If $b>\dfrac{p}{p+2m},$ then there exists $\phi(b),$ such that $$h_x^{\prime}(x,b)\le 0, \:\: 0\le x< \phi(b) \:\:\:\: \text{and} \:\:\:\: h_x^{\prime}(x,b)\ge 0, \:\: \phi(b)\le x\le 1.$$
	In this case, we need to compare the  values  of the function $h$ at $x=0$ and $x=1$:
	$$h(0,b)=(m+p)b, \quad h(1,b)=p.$$
	Thus, we see that if $b\le \dfrac{p}{m+p},$ then the function $h$ attains its maximum at the point $x=1.$ Therefore,
	$$S_rf \le \pi pr^{2p}, \quad r^{2m}\le\frac{p}{m+p}.$$
For $b\in I_{1}$  the function $h(x,b)$ has the same form as for $b\in I_{0}.$ Therefore, the function $h(x,b), \: b\in I_{1}, $ reaches its maximum at the point $x=0.$ Thus,
	$$ S_rf\le \pi (m+p)r^{2(m+p)}, \quad \dfrac{p}{m+p}\le r^{2m} \le \dfrac{m+p}{2m+p}.$$

Suppose that the function $h(x,b), \:b\in I_{s-1}, \: s\ge 2,$ reaches its maximum at the point $x=0.$ Let us prove that 
\begin{equation}\label{h_sless}
	h(x,b)\le h\Big(x, \frac{(s-1)m+p}{sm+p}\Big)+ h(0,b)-h\Big(0, \frac{(s-1)m+p}{sm+p}\Big), \quad 0\le x \le 1, \: b\in I_s. 
	\end{equation}
 If \eqref{h_sless} is true, then the function $h(x,b),$ $b\in I_s,$ reaches its maximum at the point $x=0$ as well. From this the result of theorem  follows. 

\vspace{2mm}

Let
$$k(x,b):=\dfrac{(1-x)^2(m+p(1-bx))}{(1-bx)^2}, \quad 0\le x\le 1, \: 0\le b \le 1.$$
Consider the function
$$\widetilde{k}(x,b):=\dfrac{(1-bx)^3k_x^{\prime}(x,b)}{1-x}=-pb^2x^2+pb(3-b)\cdot x-2m-2p+pb+2bm.$$

Let us note that $\widetilde{k}_x^{\prime}(x,b)\ge 3pb(1-b)>0$ and
$\dfrac{\widetilde{k}(1)}{2}=(b-1)\cdot m-(b-1)^2\cdot p\le 0.$
 Hence, $\widetilde{k}(x,b)\le 0$ and $k_x^{\prime}(x, b)\le 0, \: 0\le x \le 1, \: 0\le b\le 1.$
Note that
\begin{equation}\label{h'}
	h^{\prime}_x(x,b)=p+k^{\prime}_x(x,b)\dfrac{sm+p}{m+p}b^s\le p+k^{\prime}_x(x,b)\dfrac{m}{m+p}\dfrac{b^{s+1}}{1-b}, \quad b\in I_s, \:s\ge 2.	\end{equation}
We fix $x\in [0,1]$ and denote
$$\phi_s(b)=k^{\prime}_x(x,b)\dfrac{b^{s+1}}{1-b}, \quad b\in I_s, \: s\ge 2.$$
Let us prove that $\phi_s(b)$ is non-increasing function:
$$
	\phi_s^{\prime}(b)\le 0 \quad \Longleftrightarrow \quad \dfrac{b(1-b)k^{\prime \prime}_{xb}(x,b)(1-bx)^3}{1-x}+\dfrac{(s(1-b)+1)k^{\prime}_x(x,b)(1-bx)^3}{1-x}\le 0. \\
$$
But
\begin{equation*}
\begin{aligned}
  &\dfrac{b(1-b)k^{\prime \prime}_{xb}(x,b)(1-bx)^3}{1-x}+\dfrac{(s(1-b)+1)k^{\prime}_x(x,b)(1-bx)^3}{1-x} 
\\
\\
&=b(1-b)\big(\widetilde{k}^{\prime}_b(x,b)+3xk^{\prime}_x(x,b)(1-bx)^2\big)+(s(1-b)+1)\widetilde{k}(x,b)
\\
\\
&\le  b(1-b)\widetilde{k}^{\prime}_b(x,b)+(s(1-b)+1)\widetilde{k}(x,b)
\\
\\
&=\Big(\big((s+2)b-s-3\big)b^2x^2+\big((2+s)b^2+(-6-4s)b+3s+6\big)\cdot x
\\
\\
&+(-s-1)b^2+(2+3s)b-2s-2\Big)\cdot p -2(1-b)^2(1+s)\cdot m
\\
\\
&\le  
\Big(\big((s+2)b-s-3\big)b^2x^2+\big((2+s)b^2+(-6-4s)b+3s+6\big)\cdot x
\\
\\
&+(-s-1)b^2+(2+3s)b-2s-2\Big)\cdot p
.
\end{aligned}
\end{equation*}
Let 
\begin{multline*} w(x)=\big((s+2)b-s-3\big)b^2x^2+\big((2+s)b^2+(-6-4s)b+3s+6\big)\cdot x
	\\
	\\
	+(-s-1)b^2+(2+3s)b-2s-2.
\end{multline*}
The maximum point of the  function $w$ is
$$x_0=\dfrac{(2+s)b^2+(-6-4s)b+3s+6}{2b(s+3-(s+2)b)}\ge 1.$$
Thus, we need to calculate the value of the function $w$ at point $x=1:$
\begin{equation*}
	\begin{aligned}
&w(1)= (2+s)b^3-(5+3s)b^2+(4+3s)b-s-1.
\end{aligned}
\end{equation*}
If $b\le \dfrac{s+1}{s+2}, \: s\ge 2,$ then the last expression is non-positive. By the conditions of the theorem, $b\le \dfrac{sm+p}{(s+1)m+p}, \:s\ge 2$ and $p\le m.$ Therefore, $b\le \dfrac{s+1}{s+2}, \: s\ge 2.$
Thus, $w(x)\le 0, \: 0\le x\le 1.$ 
Therefore, $\phi_s(b), s\ge 2,$ is non-increasing function. Using this and \eqref{h'}, we conclude that
$$h^{\prime}_x(x,b)\le p+\dfrac{m}{m+p}\phi_s\Big(\dfrac{(s-1)m+p}{sm+p}\Big)=h^{\prime}_x\Big(x, \dfrac{(s-1)m+p}{sm+p}\Big), \quad b\in I_s, \: s\ge 2.$$
Hence,
\begin{equation*}
	\begin{aligned}
		&\Big(h(x,b)-h\Big(x,\dfrac{(s-1)m+p}{sm+p}\Big)\Big)^{\prime}_x\le 0, \quad 0\le x\le 1, \: b\in I_s \\ 
	\\ & \Longrightarrow \quad  h(x,b)-h\Big(x,\dfrac{(s-1)m+p}{sm+p}\Big)\le h(0,b)-h\Big(0,\dfrac{(s-1)m+p}{sm+p}\Big), \quad 0\le x\le 1, \: b\in I_s,
	\end{aligned}
\end{equation*}
which is exactly \eqref{h_sless}. Thus, the theorem is proved.
	\end{proof}
	
	\vspace{2mm}
		\section{Weighted majorant series estimates} \label{Sec4}

In this Section we will give new estimates for the weighted majorant series of the form $$\sum_{n \ge m}c_n|a_n|r^{n}, \quad c_n>0\: (n\ge 1), \: m\in \mathbb{N}_0.$$ The obtained result will be applied to the problem on the Bohr radius of the differentiation operator.
To do this, we will use  the concepts of the Bohr radius of a pair of operators   and related concepts introduced in the work \cite{Khas}. Let us recall these definitions.

Let $D$ be an open disk with center at zero or an interval with center at zero. Let us denote by
$\mathcal{H}(D)$  the set of all functions of the form $f(z)=\sum_{n\ge 0}a_nz^n,$ where the series converges in $D$. Let
$$\mathcal{H}_m(D):=\{f\in \mathcal{H}(D): \: f(0)=f^{\prime}(0)=...=f^{(m-1)}(0)=0\}.$$

\begin{definition}
	
	Let $m\in \mathbb{N}_0, \: t,s_1,s_2 \ge 0$ and $$T_1:\mathcal{H}_m(t\mathbb{D})\rightarrow \mathcal{H}(\mathbb{D}), \hspace{5mm} T_2: \mathcal{H}_m(-s_1,s_1) \rightarrow \mathcal{H}(-s_2,s_2)$$ are linear operators. Supremum of the set of numbers $R\in [0,\infty)$ such that
	
	\begin{equation} \label{def}
		\|T_1f\|_{\infty}\le 1 \Longrightarrow |T_2M_rf|\le 1, \quad 0\le  r\le R, \end{equation}
	will be called the Bohr radius of a pair $T_1$ and $T_2$  and denoted by $R_{T_1\rightarrow T_2}.$ If $T_1=T_2=T,$ then we write $R_{T}$. 
	
\end{definition}

The condition \eqref{def} is equivalent to the inequality
$$|T_2M_rf|\le \|T_1f\|_{\infty}.$$

Let us denote by $id_m$ the identity operator defined in the space $\mathcal{H}_m(\mathbb{D}).$  From Bohr's theorem it follows that $R_{id_0}=1/3.$ In 1962, E. Bombieri proved \cite{Bomb} that $R_{id_1}=1/\sqrt{2}.$ The problem of calculating $R_{id_m}$ remains open. It was discussed in \cite{PPS} and \cite{Bombb}

\vspace{2mm}
\begin{definition}
	
The function $$m_{T_1\rightarrow T_2}(r):=\sup_{f: \|T_1f\|_{\infty}\neq 0}\dfrac{|T_2M_rf|}{\|T_1f\|_{\infty}}$$
will be called the Bohr-Bombieri function of the operators $T_1$ and $T_2.$ If the operators $T_1$ and $T_2$ coincide and are equal to $T$, then we write $m_{T}(r).$
	
\end{definition}

It follows from the Bohr theorem that $m_{id_0}(r)=1, \:0\le r\le 1/3.$ In 1962, E. Bombieri and D. Ricci proved \cite{Bomb} that for $ r\in [1/3, 1/\sqrt{2}],$
$$m_{id_0}(r)=\dfrac{3-\sqrt{8(1-r^2)}}{r}.$$
For $r\in [1/\sqrt{2},1),$ the problem of calculating the Bohr-Bombieri function for the operator  $id_0$ remains open. Profound results related to this problem were obtained by E.~Bombieri and J.~Bourgain in 2004 (see \cite{Bombb}).

\vspace{3mm}
\begin{definition}
	
Let us consider the functions of the form $f(z)=\sum_{n\ge m}a_nz^n.$ We fix the modulus of the initial coefficient $a:=|a_m|$.
The supremum of the set of the numbers $r$  such that for all functions of the form $f(z)=\sum_{n\ge m}a_nz^n, \: |a_m|=a,$ the condition \eqref{def} is satisfied
will be called the Bohr radius of a pair of operators $T_1$ and $T_2$ with the initial coefficient $a$  and denoted by
$R_{T_1\rightarrow T_2}(a).$ We similarly define the Bohr--Bombieri function of a pair of operators with an initial coefficient:
$$ m_{T_1\rightarrow T_2}(r,a):=\sup\dfrac{|T_2M_rf|}{\|T_1f\|_{\infty}},$$
where the supremum is taken over all functions $f(z)=\sum_{n \ge m}a_nz^n,$ such that
$\|T_1f\|_{\infty}\neq 0$ and $|a_m|=a.$

\end{definition}

	\vspace{1mm}
For example (see (\cite{Bomb}), $$R_{id_0}(a)=\dfrac{1}{1+2a}, \quad 1/2<a\le 1.$$

\vspace{2mm}
The above concepts are convenient for formulating certain theorems and allow us to give them a unified proof. In addition, they can be used to formulate statements when the Bohr radius is unknown (see Theorems 6-9 of the work \cite{Khas}). For a number of examples of previously calculated Bohr radii and Bohr--Bombieri functions, see \cite{Khas}.

\vspace{3mm}
\begin{definition}
Let $f(z)=\sum_{n \ge m}a_nz^n, \: m\in \mathbb{N}_0$ and let $h(z)=\sum_{n \ge m}c_nz^n $ converges in the disk of the radius $s>0.$ The operator
$$A_hf(z)=(h\ast f)(z):=\sum_{n\ge m}c_na_nz^n,$$
acting from $\mathcal{H}_m(t\mathbb{D }),\: t>0,$ to $\mathcal{H}_m(st\mathbb{D }),$
is called the Hadamard convolution operator (see, for example, \cite{Rusch}). We will call the function $h$ a convolution function.
\end{definition}

\vspace{2mm}
Let $m\in \mathbb{N}_0$ and $h(z)=\sum_{n\ge m} c_nz^n.$ Let $S_{m,l}$ be a shift operator defined in the space $\mathcal{H}_m(D),$ namely $$S_{m,l}f(z):=z^{l}f(z), \quad f\in \mathcal{H}_m( D).$$ We will consider operators of the form \begin{equation*} \label{Amh} A^{m,l}_{h}f:=S_{m,l}(h \ast f)=\sum_{n \ge m}c_na_nz^{n+l}, \quad f(z)=\sum_{n\ge m}a_nz^n.
\end{equation*}

 In particular, examples of such operators are the differentiation operators defined in $\mathcal{H}_m(\mathbb{D}),$ that is \begin{equation} \label{defderiv}
	\partial^mf(z):=\dfrac{d^m}{dz^m}f(z)=m!\cdot S_{m,-m}\Big(\dfrac{z^m}{(1-z)^{m+1}} \ast f(z)\Big), \quad \partial:=\partial^1,
\end{equation}
 and the integration operator  acting from
 $\mathcal{H}_0(\mathbb{D})$ to $\mathcal{H}_1(\mathbb{D })$:
\begin{equation*}\label{int}
	\int_0^zf(\zeta)d\zeta=S_{0,1} \Big(\dfrac{-\log(1-z)}{z}\ast f(z)\Big).	
\end{equation*}

\vspace{3mm}
Let us formulate the main result of \cite{Khas} (more precisely, its special case):

\newtheorem*{D}{Theorem D}
\begin{D} \label{thB}
	Let $h(z)=\sum_{n\ge m}c_nz^n, \: c_n>0.$ If $\dfrac{r}{a}\le \inf_{n\ge m+1}\dfrac{c_{n}}{c_{n+1}}$ and $a>r,$ then 
	$$m_{id_m \rightarrow A^{m, l}_{h}}(r,a)=r^{m+l}c_ma+(1/a-a)a^{-m}r^l(h(ar)-c_m(ar)^m).$$
\end{D}

\vspace{3mm}
The main goal of this Section is to give estimates for $m_{id \rightarrow A^{m,l}_h}(r,a)$ where $r$ and $a$ such that $\dfrac{r}{ a}>\inf_{n\ge m+1}\dfrac{c_{n}}{c_{n+1}}$ and $a>r.$

\vspace{4mm}

\begin{theorem} \label{mainth}

		Let $s\ge 2, \: h(z)=\sum_{n\ge m}c_nz^n, \: \: c_n>0 \: (n\ge m+1), \: a>r$ and
	\begin{equation} \label{cond0}
		r/a\le c_{n+m}(r/a)^{n} \le c_{m+s}(r/a)^{s}, \quad 1\le n<s.
	\end{equation}
	If 
	\begin{equation} \label{cond1}
		\inf_{n\ge s+m-1}\dfrac{c_{n}}{c_{n+1}} \le \dfrac{r}{a}\le \inf_{n\ge s+m}\dfrac{c_{n}}{c_{n+1}},
	\end{equation}
	then
	\begin{multline*}
		(1/a-a)a^{-m}r^l(h(ar)-c_m(ar)^m)\le  m_{id \rightarrow A^{m,l}_{h}}(r,a)-r^{m+l}c_ma 
		\\
		\\
		\le\Big(\dfrac{c_{m+s}}{c_{m+1}}\Big(\dfrac{r}{a}\Big)^{s-1}\Big)^{1/2}(1/a-a)a^{-m}r^l(h(ar)-c_m(ar)^m).
	\end{multline*}
	For $s=2$ and $c_{m+1}=1$, the condition \eqref{cond0} is unnecessary. 
	
\end{theorem}

\vspace{3mm}
Let us first prove the following lemma:
\newtheorem*{L1}{Lemma 1}
\begin{L1}
	Let $s\ge 2, \: m\ge 0, \:h(z)=\sum_{n\ge m}c_nz^n, \: c_n> 0,\: n\ge m+1,$ \: $\inf_{n\ge s+m-1}\dfrac{c_{n}}{c_{n+1}} \le x^2\le \inf_{n\ge s+m}\dfrac{c_{n}}{c_{n+1}}$ and
	\begin{equation} \label{cond}
		x^{2}\le c_{n+m}x^{2n} \le c_{m+s}x^{2s}, \quad 1\le n<s.
	\end{equation}
	Let $f(z)=\sum_{n\ge m}a_nz^n, \: \|f\|_{\infty}\le 1, \: a=|a_m|.$ Then
	$$\sum_{n\ge m+1}c_n|a_n|^2x^{2n}\le \dfrac{c_{m+s}}{c_{m+1}}x^{2(s-1)}(1/a-a)^2a^{-2m}(h((ax)^2)-c_m(ax)^{2m}).$$ 
For $s=2$ and $c_{m+1}= 1$, the condition \eqref{cond} is unnecessary.
\end{L1}
\begin{proof}[Proof.]
By Schwarz lemma $\Big|\dfrac{f(z)}{z^m}\Big|\le 1, z\in \mathbb{D},$ therefore
$\dfrac{f(z)}{z^m}\prec \dfrac{z+a}{1+az},$ that is
$$\sum_{n\ge 0}a_{n+m}z^n \prec a+\Big(a-\dfrac{1}{a}\Big)\sum_{n\ge 1}(-az) ^n.$$
Applying  Theorem \ref*{G(r)}
to these functions and weights $\Big(\dfrac{c_{n+m}}{c_{m+1}}\Big)_{n\ge 1},$ we obtain the statement of the lemma.
\end{proof}

\vspace{3mm}
\begin{proof}[Proof of Theorem \ref{mainth}.]
To obtain the lower bound   it is enough to consider the function $f(z)=\dfrac{z+a}{1+az}.$ Let us give the upper estimate for the sum $A_h^{m,l}M_rf$ using the Cauchy--Bunyakovsky inequality. We choose $\rho=\sqrt{r/a}$ (this imposes the restriction $a>r$):
\begin{equation*}\label{chain}
	r^{-l}A_h^{m,l}M_rf=\sum_{n\ge m}c_n|a_n|r^{n}  \le
	c_mar^{m}+\Big(\sum_{n\ge m+1}c_n|a_n|^2\rho^{2n}\Big)^{1/2}\Big({\sum_{n\ge m+1}c_n(r/\rho)^{2n}}\Big)^{1/2}.
\end{equation*}
Using Lemma 1 we conclude that if
$$\inf_{n\ge s+m-1}\dfrac{c_{n}}{c_{n+1}} \le \rho^2=\dfrac{r}{a}\le \inf_{n\ge s+m}\dfrac{c_{n}}{c_{n+1}}$$ and
$$
\rho^{2} \le c_{n+m}\rho^{2n} \le c_{m+s}\rho^{2s}, \quad 1\le n<s,
$$then
\begin{multline*}
	c_mar^{m}+\Big(\sum_{n\ge m+1}c_n|a_n|^2\rho^{2n}\Big)^{1/2}\Big({\sum_{n\ge m+1}c_n(r/\rho)^{2n}}\Big)^{1/2} 
	\\ \\	\le 
	c_mar^{m}+(c_{s+m}\rho^{2(s-1)})^{1/2}(1/a-a)a^{-m}\big({h((a\rho)^2)-c_m(a\rho)^{2m}}\big)^{1/2}\cdot \\ \\ \cdot \big({(h((r/\rho)^2)-c_m(r/\rho)^{2m}}\big)^{1/2}. 
\end{multline*}

Thus, we proved the upper bound for $A_h^{m,l}M_rf.$
\end{proof}
\vspace{3mm}

In the work \cite{GolDeriv} it was shown that $R_{id \rightarrow \partial}=1-\sqrt{2/3}.$ From the proof of this theorem we can obtain the following estimates for the Bohr radius with initial coefficient:
\begin{equation*} \label{old_est}
	  1-\sqrt{\dfrac{1+a}{2+a}} \le  R_{id_1 \rightarrow \partial}(a)\le \dfrac{1}{a}\Big(1-\sqrt{\dfrac{1+a}{1+2a}}\Big), \quad 0\le a<1.
\end{equation*}
This fact was obtained using the standard estimate $|a_n|\le 1-a^2, \: n\ge 1,$ which follows from the Schwarz--Pick inequality and the symmetrization of the analytic function (see, for example, \cite{Khas}): let $\|f\|_{\infty}\le 1,$ then
$$M_rf^{\prime}=\sum_{n\ge 1}n |a_n|r^{n-1}\le a+\Big(\dfrac{1}{(1-r)^{2}}-1\Big)(1-a^2).$$
The last expression is not greater than one if $r\le 1-\sqrt{\dfrac{1+a}{2+a}},$ which proves the required lower estimate. The upper bound can be obtained by considering the functions $f_a(z)=\dfrac{z-a}{1-az}$. Note that problems on the Bohr radius of the differentiation operator and related problems were also discussed in the work \cite{BhDasDiff}.

\vspace{3mm}
In the work  \cite{Khas} for some $a$ the sharp estimate for the Bohr radius of the differentiation operator with a fixed initial coefficient $a$ was found, namely the following theorem was proved:
\newtheorem*{E}{Theorem E}
\begin{E}\label{Th:3}
		Let $m\ge 1.$ If \begin{equation}\label{condBohrDiff}
		\Big(1-\dfrac{2a^2}{m+2}\Big)^{m+1}\le \dfrac{1+a}{1+2a}, \quad 0<a<1,
	\end{equation}
	then
	$$R_{id_m\rightarrow \partial^m/m!}(a)=r_m(a):=\dfrac{1}{a}\Big(1-\sqrt[m+1]{\dfrac{1+a}{1+2a}}\Big).$$
	In other words, if $ f\in \mathcal{H}_m$ and $\Big|\dfrac{f^{(m)}(0)}{m!}\Big|=a,$ then
	$$M_r f^{(m)} \le m!\|f\|_\infty, \quad r\le r_m(a),$$
	and the number $r_m(a)$ is the best possible.
\end{E}

In particular, if $a\in (a_0,1],$ where $a_0=0.429782...$ is the root of the equation $\Big(1-\dfrac{2}{3}a^2\Big)^ {2}=\dfrac{1+a}{1+2a}, \: 0<a<1,$
then
$$R_{id_1\rightarrow \partial}(a)=\dfrac{1}{a}\Big(1-\sqrt{\dfrac{1+a}{1+2a}}\Big).$$

Using theorem \ref{mainth}, we obtain the new lower estimates for the value $R_{id_1\rightarrow \partial}(a)$ for all $a\in [0,1]$.

\vspace{2mm}
\begin{theorem} \label{BohrDiff}
	If $a\in [0, 0.429782...],$ 
	then
	$$\dfrac{as}{1+s}\le R_{id_1 \rightarrow \partial}(a)\le \dfrac{1}{a}\Big(1-\sqrt{\dfrac{1+a}{1+2a}}\Big),$$
where $s$ is the root of the equation
	\begin{equation} \label{urav}
		\Big(\dfrac{[s]+1}{2}\Big(\dfrac{s}{s+1}\Big)^{[s]-1}\Big)^{1/2}\Big(\dfrac{1}{(1-a^2\frac{s}{s+1})^2}-1\Big)=\dfrac{a}{1+a}.
	\end{equation}
	\end{theorem}
	
	\begin{proof}[Proof.]
	Let us first prove that the root of the equation \eqref{urav} exists and is unique. Let us denote
	$$g(t):=			a+\Big(\dfrac{[t]+1}{2}\Big(\dfrac{t}{t+1}\Big)^{[t]-1}\Big)^{1/2}(1/a-a)\Big(\dfrac{1}{(1-a^2\frac{t}{t+1})^2}-1\Big).$$
It is evident that the function $g$ is monotonic on every interval $[n, n+1),$ since the function $\frac{t}{t+1}$ is monotonic for $t>0.$ 	Let us show that $g$ is a continuous function. Let $p\in \mathbb{N}.$ Then
	\begin{multline*}
		\lim_{t\rightarrow p-}g(t)= a+ \Big(\dfrac{p}{2}\Big(\dfrac{p}{p+1}\Big)^{p-2}\Big)^{1/2}(1/a-a)\Big(\dfrac{1}{(1-a^2\frac{p}{p+1})^2}-1\Big)\\
		= a+\Big(\dfrac{p+1}{2}\Big(\dfrac{p}{p+1}\Big)^{p-1}\Big)^{1/2}(1/a-a)\Big(\dfrac{1}{(1-a^2\frac{p}{p+1})^2}-1\Big)=\lim_{t\rightarrow p+}g(t).
	\end{multline*}
Since the function $g$ is monotonic on every interval $[n, n+1)$ and is continuous, then $g$ is monotonic on $[1, \infty).$ Moreover, $\lim_{t\rightarrow +\infty}g(t)=+\infty$ and
$$g(2)\le 1 \iff a\le 0.429782...$$
Therefore, there is only number $s \in \mathbb{R}$ such that $g(s)=1.$
	
	\vspace{2mm}
Now we obtain the lower estimate for $R_{id_1 \rightarrow \partial}(a).$ Recall the following: $$\partial f=S_{1,-1}(h\ast f),$$
$$h(z)=\sum_{n \ge 1}c_nz^n=\sum_{n \ge 1}nz^n=\dfrac{z}{(1-z) ^{2}}.$$ It is easy to check that if $x\ge \dfrac{[s]}{[s]+1},$ then for $n\le [s]$ the sequence $nx^n$ is increasing. Therefore, by Theorem \ref{mainth}, if
	$$\inf_{n\ge [s]}\dfrac{c_n}{c_{n+1}}\le \dfrac{r}{a} \le \inf_{n\ge [s]+1}\dfrac{c_n}{c_{n+1}},$$ that is, 
	$$\dfrac{[s]}{[s]+1}\le \dfrac{r}{a} \le  \dfrac{[s]+1}{[s]+2},$$ then
	\begin{equation} \label{qwe}
		m_{id_1\rightarrow \partial}(r,a)\le a+\Big(\dfrac{[s]+1}{2}\Big(\dfrac{r}{a}\Big)^{[s]-1}\Big)^{1/2}(1/a-a)\Big(\dfrac{1}{(1-ar)^2}-1\Big).
	\end{equation}
By the  definition of the number $s$, the right-hand side of \eqref{qwe} is not greater than one when $r\le \dfrac{as}{s+1}.$ Thus, we proved the lower bound for $R_{id_1\rightarrow \partial}(a)$.

	\end{proof}
\vspace{4mm}

 Let us compare the old and the new methods for calculating the lower bound for $R_{id_1\rightarrow \partial}(a).$ Let us denote the old lower bound by $R_{id_1\rightarrow \partial}(a)$ by
 $$\tilde{r}(a)=1-\sqrt{\dfrac{1+a}{2+a}},$$
 and the upper estimate by
 $$ \hat{r}(a)=\dfrac{1}{a}\Big(1-\sqrt{\dfrac{1+a}{1+2a}}\Big).$$
 For fixed $a$, we choose the corresponding $s$ defined in the Theorem \ref{BohrDiff} and the corresponding new lower bound $r_s(a).$ We present a table for these quantities:
 
 \vspace{4.5mm}

 \begin{center}
 	\begin{tabular}{|c|c|c|c|c|} 
 		\hline
 		$a$& $\tilde{r}(a)$ & $[s]$ &$r_s(a)$ &$\hat{r}(a)$
 		\\
 		\hline
 		\: \: 0.429  \: \:&   \: \: 0.23298...  \:  \:&  \: \: 2  \: \:&  \:  \:0.28652...   \:\:&  \:  \:0.28674...   \:\: \\ 
 		\hline
 		0.42 & 0.23398... & 2 & 0.28656...&0.28931...\\ 
 		\hline
 		0.4& 0.23623... & 2 & 0.28651... &0.29520...\\ 
 		\hline
 		0.38& 0.23853... & 3 & 0.28598... & 0.30134...\\ 
 		\hline
 		0.36& 0.24087... & 3 & 0.28210...&0.30774...\\ 
 		\hline
 		
 		0.34& 0.24326... & 4 & 0.27674... & 0.31442...\\ 
 		\hline		
 		
 		0.32& 0.24570... &5 & 0.26935... &0.32140...\\ 
 		\hline	
 		
 		0.3& 0.24819... &6 & 0.26002... &0.32870...\\ 
 		\hline	
 		
 		0.28& 0.25073... &8 & 0.24899... &0.33635...\\ 
 		\hline	
 		
 		0.26& 0.25332... &9 & 0.23631... &0.34436...\\ 
 		\hline

 	\end{tabular}
 \end{center}
 
 \vspace{4.5mm}
 
 Thus, we see that in some cases the new method is  more effective than the known one. Moreover, the Theorem \ref{BohrDiff} allows us to formulate the conjecture that if in the problem under consideration there is only the restriction $a>r,$ then the Bohr radius equals to the known upper bound:
 
 \newtheorem*{conj}{Conjecture}
 \begin{conj}
 	If $a\in (a_1,1],$ where $a_1=0.321037...$ is the root of equality
 	$1-\sqrt{\dfrac{1+a}{1+2a}}=a^2,$
 	then
 		$$R_{id_1\rightarrow \partial}(a)=\dfrac{1}{a}\Big(1-\sqrt{\dfrac{1+a}{1+2a}}\Big).$$
 	\end{conj}

\vspace{5mm}

\textbf{Acknowledgments.}  
The author thanks A.D. Baranov for his attention to the work on the
article. I also thanks the referee for careful reading of my paper  and for useful comments.

The results of Section 2 and Section 3 were supported by Ministry of Science and Higher Education of the Russian Federation (agreement No 075-15-2024-631) and Theoretical Physics and Mathematics Advancement
Foundation ``BASIS”. The results of Sections 4 were obtained with the support of Russian Science Foundation project 19-71-30002-$\Pi$.


\vspace{4mm}

\begin{thebibliography}{1}
	
	
	\bibitem{Aizen}L. Aizenberg, Multidimensional analogues of Bohr’s theorem on power series, Proc. Amer.
	Math. Soc., 128, 1147–1155, (2000).
	
	\bibitem{AKP}S. Alkhaleefah, I. Kayumov, S. Ponnusamy, On the Bohr inequality with a fixed zero coefficient, Proc. Amer.
	Math. Soc., 147(12), 5263–5274, (2019).
	
\bibitem{BPS}F. Bayart, D. Pellegrino, J. Seoane-Sepúlveda,  The Bohr radius of the n-dimensional polydisk is
equivalent to
$\sqrt{\log{n}/{n}}$, Adv. Math., 274, 726-746, (2014).

	
	
	
	\bibitem{BhowDas}B. Bhowmik, N. Das, Bohr phenomenon for subordinating families of certain univalent functions, J. Math. Anal. Appl., 462(2), 1087-1098, (2018).
	
	\bibitem{BhDasDiff}B. Bhowmik,  N. Das, On some aspects of the Bohr inequality, Rocky Mount. J. Math., 51(1), (2021).
	
	\bibitem{Boas}H. Boas, Majorant series, several complex variables, J. Korean Math. Soc., 37, 321–337,
	(2000).
	
	\bibitem{BoasKhav}H. Boas, D. Khavinson, Bohr’s power series theorem in several variables, Proc. Amer. Math.
	Soc., 125, 2975–2979, (1997).
	
	\bibitem{Bohr}H. Bohr, A theorem concerning power series, Proc. Lond. Math. Soc., 13, 1-5, (1914).
	
	\bibitem{Bombb}E. Bombieri, J. Bourgain, A remark on Bohr's inequality, Int. Math. Res. Not., 80,  4307-4330, (2004).
	
	\bibitem{Bomb}E. Bombieri, Sopra un teorema di H. Bohr e G. Ricci sulle funzioni maggioranti delle serie di potenze, Boll. Un. Mat. Ital., 17(3), 276-282, (1962).
	
		\bibitem{Ann.Math.}A. Defant, L. Frerick, J. Ortega-Cerdà, M.~Ounaïes, K.~Seip, The Bohnenblust-Hille inequality for
	homogeneous polynomials is hypercontractive, Ann. Math, 174(1), 485-497, (2011).
	
	\bibitem{Dixon05}P. Dixon, Banach algebras satisfying the non-unital von Neumann inequality, Bull. Lond. Math. Soc., 27(4), (2005).
	
	
	\bibitem{AreaQuasi}P. Eenigenburg, J. Waniurski, An area inequality for quasi-subordinate analytic functions, Annal. Polon. Math., 34, 25-33, (1977).
	
	\bibitem{GrowthArea}P. Eenigenburg,  On the relative growth of area for subordinate functions, Rocky Mount. J. Math., 19(2), 415-422, (1989).
	
	\bibitem{GolDeriv}G. Goluzin, Some estimations of derivatives of bounded functions.  Sb. Math., 16(58), 295–306, (1945).
	
	\bibitem{Gol}G. Goluzin, On the majorization of subordinate analytical functions. I,  Sb. Math, 29(71), 209–224, (1951).
	
		\bibitem{IKP}A. Ismagilov,  I. Kayumov, S. Ponnusamy, Sharp Bohr type inequality, J. Math. Anal. Appl, 489(1), 124-147, (2020).
		
	\bibitem{Khas}R. Khasyanov, The Bohr radius and the Hadamard convolution operator,  J. Math. Anal. Appl., 531(1), (2024).
	
	\bibitem{KayumpLacun}I. Kayumov, S. Ponnusamy,   Bohr’s inequalities for the analytic functions with lacunary series and harmonic functions., J. Math. Anal. Appl., 465(2), 857–871, (2018).
	
	
	
	\bibitem{KayumpImpr}I. Kayumov, S. Ponnusamy, Improved version of Bohr's inequality, C. R. Acad. Sci. Paris, Ser., 356, 272-277, (2018).
	
\bibitem{Little}J. E.~Littlewood, On inequalities in the theory of functions, Proc. London Math. Soc., 23, 481–519,  (1925).
	
	
	\bibitem{PaSi}V. Paulsen, D. Singh, Bohr’s inequality for uniform algebras, Proc. Amer. Math. Soc., 132(12), 3577–3579,
	(2004).
	
		\bibitem{PVW}S. Ponnusamy,  R. Vijayakumar, K.-J. Wirths, Improved Bohr's phenomenon in quasi-subordination classes,  J. Math. Anal. Appl., 506(1), 125-645, (2021).
		
	\bibitem{PPS}V. Paulsen, G. Popescu, D. Singh, On Bohr's inequality, Proc. London Math. Soc.,  85(3), 493-512, (2002) .
	
	\bibitem{Reich}E. Reich,
	An inequality for subordinate analytic functions, Pacific J. Math., 4(2), 259-274, (1954).
	
	\bibitem{Ricci}G. Ricci, Complementi a un teorema di H. Bohr riguardante le serie di potenze., Rev. Un. Mat.
	Argentina, 17, 185–195, (1955/1956)
	
	
	
	
	\bibitem{Rusch}S. Ruscheweyh, Convolutions in geometric function theory, Seminaire de Math. Sup., 83, Presses de l’Universit\'e de Montr\'eal, (1982).
	

\end{thebibliography}
\end{document}